\documentclass[12pt]{article}
\usepackage{bbm}
\usepackage{latexsym}
\usepackage{mathrsfs}
\usepackage{graphicx}
\usepackage{epstopdf}
\usepackage{subfigure}
\usepackage{amssymb}
\usepackage{amsmath}
\usepackage{amsthm}
\usepackage{color}
\usepackage{cite}
\usepackage{float}
\usepackage{indentfirst}
\usepackage{anysize}\marginsize{45mm}{45mm}{40mm}{50mm}
\usepackage{caption}
\usepackage{float}

\usepackage{enumerate}

\newcommand{\FS}[2]{\displaystyle\frac{#1}{#2}}
\newtheoremstyle{lemma}{\topsep}{\topsep}%
     {}
     {}
     {\bfseries}
     {}
     {0.1em}
     {\thmname{#1}\thmnumber{ #2}\thmnote{ #3}}
\theoremstyle{lemma}  

\newtheorem{theorem}{Theorem}[section]    
\newtheorem{lemma}[theorem]{Lemma}
\newtheorem{corollary}[theorem]{Corollary}

\numberwithin{equation}{section}

\usepackage{latexsym, bm}
\title
{\textbf{Forcing and anti-forcing polynomials of perfect matchings of a pyrene system
\thanks{Supported by the Science and Technology Research Foundation of the Higher Education Institutions of Ningxia, China (Grant No. NGY2018-139), and
the Natural Science Foundation of Ningxia, China (Grant No. 2019AAC03124).}}}
\author{{Kai Deng$^{a,}$\thanks{
Corresponding author. \newline
\emph{E-mail address}: dengkai04@126.com, lsh1808@163.com, zhouxiangqian0502@126.com}, ~Saihua Liu$^{b}$, ~Xiangqian Zhou$^{c}$ }\\
{\footnotesize $^{a}$School of Mathematics and Information Science,
North Minzu University,}
\\{\footnotesize Yinchuan, Ningxia 750027, P. R.~China}
\\{\footnotesize $^{b}$Department of Mathematics, Wuyi University,
Jiangmen, Guangdong 529020, P. R.~China}
\\{\footnotesize $^{c}$School of Mathematics and Statistics, Huanghuai University, Zhumadian, Henan 463000, P. R.~China}
}
\date{}

\begin{document}
\maketitle
\begin{abstract}
The forcing number of a perfect matching of a graph was introduced by Harary et al.,
which originated from Klein and Randi\'{c}'s ideal  of innate degree of freedom of Kekul\'{e} structure in molecular graph.
On the opposite side in some sense,
Vuki\v{c}evi\'{c} and Trinajsti\'{e} proposed the anti-forcing number of a graph, afterwards Lei et al. generalized this idea to single perfect matching.
Recently the forcing and anti-forcing polynomials of perfect matchings of a graph were proposed as counting polynomials for perfect matchings with the same forcing number and anti-forcing
number respectively. In this paper,
we obtain the explicit expressions of forcing and anti-forcing
polynomials of a pyrene system.
As consequences, the forcing and anti-forcing spectra of a pyrene system
are determined.
\vskip 0.1 in

\end{abstract}

\section{Introduction}
Let $G$ be a simple graph with vertex set $V(G)$ and edge set $E(G)$. A {\em perfect matching} of $G$ is a set of independent edges which covers all vertices of $G$. A perfect matching coincides with a Kekul\'{e}
structure of a conjugated molecule graph (the graph representing the carbon-atoms),
Klein and Randi\'{c} \cite{Klein, Randic} discovered
the phenomenon that a Kekul\'{e} structure
can be determined by a few number of fixed double bonds,
and they defined the {\em innate degree of freedom} of a
Kekul\'{e} structure as
the minimum number of fixed double bonds required
to determine it.
Further, the sum over innate degree of freedom of all Kekul\'{e} structures
of a graph was called the \emph{degree of freedom} of the graph,
which was proposed as a novel invariant to estimate the resonance energy.
In 1991, Harary, Klein and \v{Z}ivkovi\'{c}\cite{Harary}
extended the concept ``degree of freedom" to a graph $G$ with a perfect matching,
and renamed it as the \emph{forcing number} of a perfect matching $M$,
denoted by $f(G,M)$.
Over the past 30 years,
many researchers were attracted to the study on
the forcing numbers of perfect matchings of a graph \cite{Che},
in addition, the anti-forcing number \cite{VT1, VT2, Zh2} was proposed from
the point of opposite view of forcing number.
In general, to compute the forcing number of a perfect matching of a bipartite graph with the maximum degree 3 is an NP-complete problem \cite{Adams},
and to compute the anti-forcing
number of a perfect matching of a bipartite graph with the maximum degree 4 is also an NP-complete problem \cite{ZhangD2}.
But the particular structure of a graph enables
us to do much better.
In this paper, we will calculate the forcing and
anti-forcing polynomials of a pyrene system,
as consequences, the exact values of forcing and anti-forcing
numbers of a perfect matching of the pyrene system are determined.

A {\em forcing set} $S$ of a perfect matching
$M$ of a graph $G$ is a subset of $M$ such that
$S$ is contained in no other perfect matchings of $G$.
Therefore, $f(G,M)$ equals the smallest cardinality over all forcing sets of $M$.
The {\em minimum} (resp. {\em maximum}) \emph{forcing number} of $G$ is the minimum (resp. maximum) value over forcing numbers
of all perfect matchings of $G$, denoted by $f(G)$ (resp. $F(G)$).
Afshani et al. \cite{Afshani} proved that
the smallest forcing number problem of graphs is NP-complete for bipartite graphs with maximum degree four.
In order to investigate the distribution of forcing numbers
of all perfect matchings of a graph $G$, the \emph{forcing spectrum} \cite{Adams}
of $G$ is proposed, denoted by Spec$_f(G)$, which is the collection of
forcing numbers of all perfect matchings of $G$.
Further, Zhang et al. \cite{ZhangZhao} introduced the \emph{forcing polynomial} of a graph, which can enumerate the number of perfect matchings
with the same forcing number.

A {\em hexagonal system} (or \emph{benzenoid}) is a finite 2-connected planar bipartite graph in which each interior face is
surrounded by a regular hexagon of side length one.
Hexagonal systems are extensively used in the study of benzenoid hydrocarbons \cite{Cyvin}, as they  properly represent
the skeleton of such molecules.
Zhang and Li \cite{Li} and Hansen and Zheng \cite{Hansen} characterized
independently the hexagonal systems with minimum forcing number 1,
and the forcing spectrum of such a hexagonal system was determined by Zhang and Deng \cite{ZhangD}.
Afterwards Zhang and Zhang \cite{Zhang} characterized plane elementary
bipartite graphs with minimum forcing number 1.
Xu et al. \cite{ZhangF} proved that the maximum forcing number of
a hexagonal system equals its \emph{Clar number}
(i.e. the number of hexagons in a maximum resonance set),
which is an invariant used to measure the stability of benzenoid hydrocarbons.
Similar results also hold for polyomino graphs \cite{Zhou2} and (4,6)-fullerenes \cite{Shi1}.
Zhang et al. \cite{Ye} proved that
the minimum forcing number of a fullerene graph is
not less than 3, and the lower bound can be achieved by infinitely many fullerene graphs.
Randi\'{c}, Vuki\v{c}evi\'{c} and Gutman \cite{c70, c72, c60}
determined the forcing spectra of fullerene graphs
C$_{60}$, C$_{70}$ and C$_{72}$,
in particular there is a single Kekul\'{e} structure of C$_{60}$
 that has the highest degree of freedom 10
 such that all hexagons of C$_{60}$ have three double CC bonds,
which represents the Fries structure of C$_{60}$ and is the most important valence structure.
For forcing polynomial, only a few types of hexagonal systems have been studied, such as catacondensed hexagonal systems \cite{ZhangZhao}
and benzenoid parallelogram \cite{Zhaos}.
For more results on forcing number, we refer the reader to see
\cite{Lam, Matthew, Wang, Jiang, Che1, Jiang1, Kleinerman, Pachter, Shi,
Zhou1, Zhou, Zhaos2}.

Given a perfect matching $M$ of a graph $G$.
A subset $S \subseteq E(G)\setminus M$ is called an \emph{anti-forcing set} of $M$
if $M$ is the unique perfect matching of $G-S$.
The smallest cardinality over all anti-forcing sets of $M$
is called  the \emph{anti-forcing number} of $M$, denoted by $af(G, M)$.
The minimum (resp. maximum)
anti-forcing number $af(G)$ (resp. $Af(G)$) of graph $G$ is the minimum (resp. maximum) value of
anti-forcing numbers over all perfect matchings of $G$.
The (minimum) anti-forcing number of a graph was first introduced
by Vuki\v{c}evi\'{c} and Trinajsti\'{e} \cite{VT1, VT2} in 2007-2008.
Actually, the hexagonal systems with minimum anti-forcing number 1
had been characterized by Li \cite{Li2} in 1997,
where he called such a hexagonal system has a forcing single edge.
Deng \cite{Deng1, Deng2} obtained the minimum anti-forcing numbers of benzenoid chains and double benzenoid chains.
Zhang et al. \cite{ZhangQQ} computed the minimum anti-forcing number of catacondensed phenylene.
Yang et al. \cite{Yang} showed that a fullerene graph has the minimum anti-forcing number at least 4, and characterized the fullerene graphs with
minimum anti-forcing number 4.

In 2015, Lei et al. \cite{Zh2}
generalized the anti-forcing number to single perfect matching of a graph.
By an analogous manner as the forcing number,
the \emph{anti-forcing spectrum} of a graph $G$ was proposed,
denoted by Spec$_{af}(G)$, which is the collection of
anti-forcing numbers of all perfect matchings of $G$.
Further, Hwang et al. \cite{Hwang} introduced the
\emph{anti-forcing polynomial} of a graph, which can enumerate the number of perfect matchings with the same anti-forcing number.
Lei et al. \cite{Zh2} proved that the maximum anti-forcing number of a hexagonal system equals its Fries number,
which can measure the stability of benzenoid hydrocarbons.
Analogous results were obtained on (4,6)-fullerenes \cite{Shi1}.
Further more, two tight
upper bounds on the maximum anti-forcing numbers of graphs were obtained \cite{ZhangD1, Shi2}.
The anti-forcing spectra of some types of hexagonal systems
were proved to be continuous, such as monotonic constructable hexagonal systems \cite{ZhangD2}, catacondensed hexagonal systems \cite{ZhangD3}.
Zhao and Zhang \cite{Zhaos1, Zhaos2} computed the anti-forcing polynomials of benzenoid systems with minimum forcing number 1 and some rectangle grids.

In this paper, we will calculate the forcing and anti-forcing polynomials
of a pyrene system.
In section 2, as a preparation,
some basic results on forcing and anti-forcing numbers are introduced,
and we characterize the maximum set of disjoint $M$-alternating cycles
and the maximum set of compatible $M$-alternating cycles with respect to
 a perfect matching $M$ of a pyrene system.
In section 3, we give a recurrence formula
for the forcing polynomial of a pyrene system,
and derive the explicit expressions of
forcing polynomial and degree of freedom of a
pyrene system. As corollaries,
the minimum forcing number, maximum forcing number and
the forcing spectrum of a pyrene system are determined,
and an asymptotic behavior of degree of freedom is revealed.
In section 4, we obtain a recurrence formula
for the anti-forcing polynomial of a pyrene system,
and derive the explicit expressions of
anti-forcing polynomial and the sum over the anti-forcing numbers of all
perfect matchings of a pyrene system.
As consequences,
the minimum anti-forcing number, maximum anti-forcing number and
the anti-forcing spectrum of a pyrene system are determined,
and an asymptotic behavior of
the sum over the anti-forcing numbers of all
perfect matchings of a pyrene system is obtained.

\section{Preliminaries}

Let $M$ be a perfect matching of a graph $G$.
A cycle $C$ of $G$ is called an {\em $M$-alternating cycle}
if the edges of $C$ appear alternately in $M$ and $E(G)\setminus M$.
If $C$ is an $M$-alternating cycle,
then the symmetric difference $M\triangle C$ is the another perfect matching
of $G$, here $C$ may be viewed as its edge set.
Let $c(M)$ be the maximum number of disjoint $M$-alternating cycles of $G$.
Since any forcing set of $M$ has to contain
at least one edge of each $M$-alternating cycle, $f(G,M)\geq c(M)$.
Pachter and Kim \cite{Pachter} proved the following theorem
by using the minimax theorem on feedback set \cite{Younger}.

\begin{theorem}\label{cycle}{\rm\cite{Pachter}}{\bf .}
Let $M$ be a perfect matching in a planar bipartite graph $G$. Then $f(G,M)=c(M)$.
\end{theorem}

\begin{figure}[http]
\centering
\subfigure[$H_{n}$]{
    \label{py}
 \includegraphics[height=35mm]{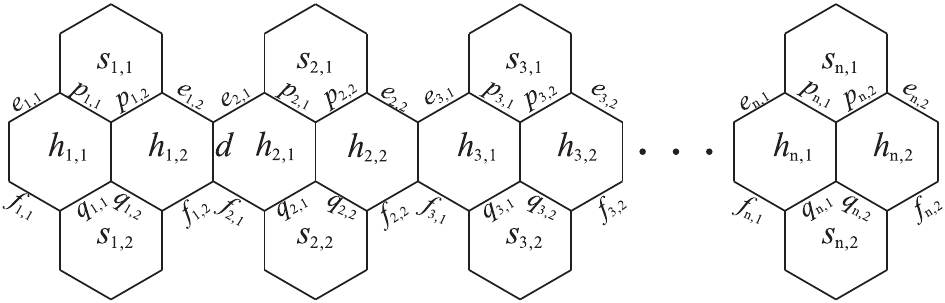}}
\vspace{5mm}
\subfigure[$G_{n}$]{
    \label{py1}
 ~~~~~~\includegraphics[height=35mm]{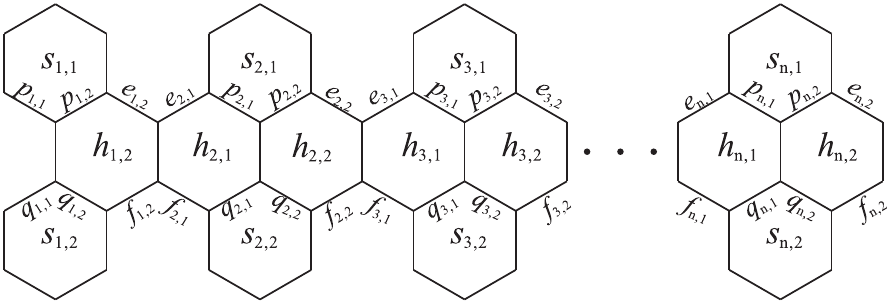}}
 \caption{Pyrene system $H_{n}$ with $n$ pyrene (a) and the auxiliary graph $G_{n}$ (b)}
\label{Pyrene}
\end{figure}

A pyrene system with $n$ pyrene fragments is denoted by $H_{n}$,
see Fig. \ref{py}. $H_{n}$ is a hexagonal system
with perfect matchings, by Theorem
\ref{cycle} $f(H_n,M)=c(M)$ for any perfect matching $M$
of $H_n$.

\begin{lemma}\label{resonance}\cite{Guo, Zhang}{\bf .}
Let $M$ be a perfect matching of a hexagonal system $H$,
$C$ an $M$-alternating cycle in $H$.
Then there is an $M$-alternating hexagon in the interior of $C$.
\end{lemma}

Let $H$ be a hexagonal system with a perfect matching $M$.
A set of disjoint $M$-alternating hexagons of $H$ is called an \emph{$M$-resonant set},
the size of a \emph{maximum $M$-resonant set} is denote by $h(M)$.

\begin{lemma}\label{feh}{\bf .}
Let $M$ be a perfect matching of
the pyrene system $H_{n}$.
Then $f(H_{n},M)=h(M)$.
\end{lemma}

\begin{proof}
Let $\mathcal{A}$ be a maximum set of disjoint $M$-alternating cycles
containing hexagons as more as possible.
By Theorem \ref{cycle}, $f(H_{n},M)=|\mathcal{A}|$.
We claim that $\mathcal{A}$ is an $M$-resonance set,
otherwise $\mathcal{A}$ contains a non-hexagonal cycle $C$.
By Lemma \ref{resonance}
there is an $M$-alternating hexagon $h$ in the interior of $C$.
Note that $\mathcal{A^{\prime}}=(\mathcal{A}\setminus\{C\})\cup\{h\}$ also is a maximum set of disjoint $M$-alternating cycles,
 but $\mathcal{A^{\prime}}$ contains more hexagons than $\mathcal{A}$, a contradiction.
We have $|\mathcal{A}|\leq h(M)\leq f(H_{n},M)=|\mathcal{A}|$,
i.e. $f(H_{n},M)=h(M)$.
\end{proof}

Let $M$ be a perfect matching of a graph $G$.
A set $\mathcal{A}$ of
$M$-alternating cycles of $G$ is called a compatible $M$-alternating set if any two cycles of $\mathcal{A}$ either are disjoint or intersect only
at edges in $M$. Let $c^{\prime}(M)$ denote the maximum cardinality over all compatible $M$-alternating sets of $G$. Since any anti-forcing set of $M$ must contain
at least one edge of each $M$-alternating cycle, $af(G,M)\geq c^{\prime}(M)$.
Lei et al. \cite{Zh2} gave the following minimax
theorem.

\begin{theorem}\label{minimax2}\cite{Zh2}{\bf .}
Let $G$ be a planar bipartite graph with a perfect matching $M$. Then
$af(G,M)=c^{\prime}(M)$.
\end{theorem}

Let $G$ be a plane bipartite graph with a perfect matching $M$. Given a compatible $M$-alternating set $\mathcal{A}$, two cycles $C_1$ and $C_2$
of $\mathcal{A}$ are \emph{crossing} if they share an edge $f$ in $M$
and the four edges adjacent to $f$ alternate in $C_1$ and $C_2$
(i.e., $C_1$ enters into $C_2$ from
one side and leaves for the other side via $f$).
A compatible $M$-alternating set $\mathcal{A}$ is called
\emph{non-crossing} if any two cycles of $\mathcal{A}$ are non-crossing.

\begin{lemma}\label{minimax9}\cite{Zh2,ZhangD1}{\bf .}
Let $G$ be a plane bipartite graph with a perfect matching $M$.
Then there is a non-crossing compatible $M$-alternating set $\mathcal{A}$
such that $|\mathcal{A}|=c^{\prime}(M)$.
\end{lemma}

A triphenylene is a benzenoid consisting of
four hexagons, one hexagon at the center,
for the other three disjoint hexagons,
each of them has a common edge with the center one.
For example, the four hexagons $s_{1,1},s_{1,2},h_{1,2},h_{2,1}$
form a triphenylene, see Fig. \ref{py}.

\begin{lemma}\label{minimax3}{\bf .}
Let $M$ be a perfect matching of the pyrene system $H_n$.
Then there is a maximum non-crossing compatible $M$-alternating set $\mathcal{A}$ such that each member of $\mathcal{A}$ either is a hexagon or the periphery of a triphenylene.
\end{lemma}

\begin{proof}
By Lemma \ref{minimax9},
there is a maximum non-crossing compatible $M$-alternating set $\mathcal{A}$
with $I(\mathcal{A})=\sum_{C\in \mathcal{A}}I(C)$ as small as possible,
where $I(C)$ denotes the number of hexagons in the interior of $C$.
Let $C$ be a member of $\mathcal{A}$.
Suppose $C$ is not a hexagon,
by Lemma \ref{resonance},
there is an $M$-alternating hexagon $h$ in the interior of $C$.
Note that $C$ and $h$ must be compatible,
otherwise $\mathcal{A}^\prime=(\mathcal{A}\backslash\{C\})\cup\{h\}$ can be a maximum non-crossing compatible $M$-alternating set such that $I(\mathcal{A}^\prime)<I(\mathcal{A})$,
a contradiction. In fact, $C$ has to be compatible with any
$M$-alternating hexagon, which implies that $h:=h_{i,j}$,
without loss of generality, let $h:=h_{i,1}(i\neq1)$ (see Fig. \ref{py}).
Then $e_{i,1}, f_{i,1}$ and the right vertical edge of $h_{i,1}$
all belong to $M$. Let $M^\prime=M\triangle h_{i,1}$.
Then $s_{i,1}$ and $s_{i,2}$ both are $M^\prime$-alternating hexagons.

\textbf{Claim 1.} $h_{i-1,2}$ also is $M^\prime$-alternating.
\begin{proof}
Suppose $h_{i-1,2}$ is not $M^\prime$-alternating.
Then at least one of $p_{i-1,2}$ and $q_{i-1,2}$
does not belong to $M$.
If only one of $p_{i-1,2}$ and $q_{i-1,2}$ belongs to $M$,
say $p_{i-1,2}\in M$, then $s_{i-1,2}$ is an $M$-alternating hexagon
which is not compatible with $C$, a contradiction.
Therefore both of $p_{i-1,2}$ and $q_{i-1,2}$ are not in $M$,
then $h_{i-1,1}$ is $M$-alternating.
If $p_{i-2,2}$ and $q_{i-2,2}$ both belong to $M$,
then the four hexagons $h_{i-2,2}$, $h_{i-1,1}$, $s_{i-1,1}$, and $s_{i-1,2}$
form a triphenylene whose periphery $T$ is an $M$-alternating cycle.
Note that $T$ is compatible with each cycle of $\mathcal{A}\setminus\{C\}$,
thus  $(\mathcal{A}\setminus\{C\})\cup\{T\}$ can be a maximum non-crossing
compatible $M$-alternating set with $I((\mathcal{A}\setminus\{C\})\cup\{T\})<I(\mathcal{A})$,
a contradiction. Hence at least one of $p_{i-2,2}$ and $q_{i-2,2}$ does
not belong to $M$, by similar  discussion as in the previous steps,
we can show that $h_{i-2,1}$ is $M$-alternating.
Keeping on the process, and finally we will prove that
$h_{1,1}$ is $M$-alternating,
however $h_{1,1}$ is not compatible with $C$, a contradiction.
\end{proof}

According to Claim 1 and the minimality of $I(\mathcal{A})$,
$C$ has to be the periphery of the triphenylene consisting of
the four hexagons $h_{i-1,2}$, $h_{i,1}$, $s_{i,1}$ and $s_{i,2}$
(see Fig. \ref{py}).
\end{proof}

\section{Forcing polynomial of pyrene system}

 The forcing polynomial of a graph
$G$ is defined as follow \cite{ZhangZhao}:

\begin{equation}\label{eq1}
F(G,x)=\sum_{M\in \mathcal{M}(G)} x^{f(G,M)}
=\sum_{i=f(G)}^{F(G)} w_ix^{i},
\end{equation}
where $\mathcal{M}(G)$ is the collection of all perfect matchings of $G$,
$w_i$ is the number of perfect matchings of $G$ with the forcing number $i$.

As a consequence,
let $\Phi(G)$ be the number of perfect matchings of a graph $G$,
then $\Phi(G)=F(G,1)$.
Recall that the degree of freedom of a graph $G$
is the sum over the forcing numbers of all perfect matchings of $G$,
denoted by $IDF(G)$,
then $IDF(G)=\frac{d}{dx}F(G,x)|_{x=1}$.
$\Phi(G)$ and $IDF(G)$ both are chemically meaningful indices within a resonance theoretic context \cite{Klein, Randic}.
Note that if $G$ is a null graph or a graph has a unique perfect matching,
then $F(G,x)=1$.

\begin{figure}[http]
\centering
\subfigure[$H_1$]{
    \label{pyrene2}
 \includegraphics[height=30mm]{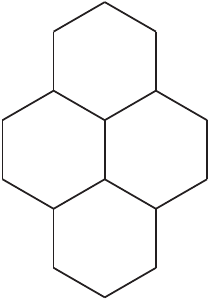}}
\hspace{30mm}\subfigure[$L$]{
    \label{phenanthrene}
 \includegraphics[height=30mm]{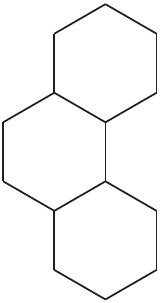}}
 \hspace{30mm}\subfigure[$N$]{
    \label{diphenyl}
 \includegraphics[height=30mm]{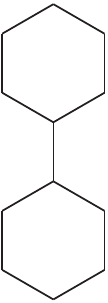}}
 \caption{Pyrene (a), Phenanthrene (b) and Diphenyl (c)}
\label{dpp}
\end{figure}

In the following we want to
derive a recurrence formula for forcing polynomial of a pyrene system,
as preparations the forcing polynomials of pyrene,  phenanthrene and diphenyl
are computed: $F(H_1,x)=4x^2+2x$, $F(L,x)=4x^2+x$, $F(N,x)=4x^2$ (see Fig. \ref{dpp}).

\begin{theorem}\label{main3}{\bf .}
Let $H_n$ be a pyrene system with $n$ pyrene fragments.
Then
\begin{equation}\label{eq99}
F(H_n,x)=(4x^2+2x)F(H_{n-1},x)-x^2F(H_{n-2},x),
\end{equation}
where $n\geq2$, $F(H_0,x)=1$ and $F(H_1,x)=4x^2+2x$.
\end{theorem}

\begin{proof}
First we introduce an auxiliary graph $G_n$ obtained by
deleting the leftmost hexagon $h_{1,1}$ from $H_n$, see Fig. \ref{py1}.
We divide $\mathcal{M}(H_n)$ in two subsets:
$\mathcal{M}_{f_{1,2}}^{e_{1,2}}(H_n)=\{M\in\mathcal{M}(H_n)\mid e_{1,2},f_{1,2}\in M\}$, $\mathcal{M}_{\bar{f}_{1,2}}^{\bar{e}_{1,2}}(H_n)=\{M\in\mathcal{M}(H_n)\mid e_{1,2},f_{1,2}\not\in M\}$.
If $M\in \mathcal{M}_{f_{1,2}}^{e_{1,2}}(H_n)$,
then $h_{1,2}$ is a unique $M$-alternating hexagon in
the leftmost pyrene fragment,
and $M^{\prime}=M\cap E(G_{n-1})$ is a perfect matching of the graph $G_{n-1}$ obtained by deleting vertices of the leftmost pyrene fragment and their incident edges from $H_n$.
By Lemma \ref{feh}, $f(H_n,M)=f(G_{n-1},M^{\prime})+1$.
If $M\in\mathcal{M}_{\bar{f}_{1,2}}^{\bar{e}_{1,2}}(H_n)$,
then the restriction $M_1$ of $M$ on the phenanthrene $L$ consisting of three hexagons $s_{1,1},h_{1,1},s_{1,2}$ is a perfect matching of $L$,
and $M_2=M\cap E(H_{n-1})$ is a perfect matching of the subsystem $H_{n-1}$
obtained by deleting vertices of $L$ and their incident edges from $H_n$,
see Fig. \ref{py1}.
According to Lemma \ref{feh}, $f(H_n,M)=f(L,M_1)+f(H_{n-1},M_2)$.
By Eq. $(\ref{eq1})$, we have

\begin{align}\label{eq11}
\nonumber F(H_n,x) & =\sum_{M\in \mathcal{M}(H_n)} x^{f(H_n,M)}\\
\nonumber & =\sum_{M\in \mathcal{M}_{f_{1,2}}^{e_{1,2}}(H_n)} x^{f(H_n,M)}+
\sum_{M\in \mathcal{M}_{\bar{f}_{1,2}}^{\bar{e}_{1,2}}(H_n)} x^{f(H_n,M)}\\
\nonumber & =\sum_{M^{\prime}\in \mathcal{M}(G_{n-1})} x^{f(G_{n-1},M^{\prime})+1}
+\sum_{M_{1}\in\mathcal{M}(L),M_{2}\in\mathcal{M}(H_{n-1})} x^{f(L,M_{1})+f(H_{n-1},M_2)}\\
\nonumber & =x\sum_{M^{\prime}\in \mathcal{M}(G_{n-1})} x^{f(G_{n-1},M^{\prime})}
+\sum_{M_{1}\in\mathcal{M}(L),M_{2}\in\mathcal{M}(H_{n-1})} x^{f(L,M_{1})}x^{f(H_{n-1},M_2)}\\
\nonumber & =xF(G_{n-1},x)+(\sum_{M_{1}\in\mathcal{M}(L)}x^{f(L,M_{1})})
(\sum_{M_{2}\in\mathcal{M}(H_{n-1})}x^{f(H_{n-1},M_2)})\\
\nonumber& =xF(G_{n-1},x)+F(L,x)F(H_{n-1},x)\\
& =xF(G_{n-1},x)+(4x^2+x)F(H_{n-1},x).
\end{align}

Now we deduce a recurrence  relation for forcing polynomial of
the auxiliary graph $G_n$.
We can divide $\mathcal{M}(G_{n})$ in two types,
one is perfect matchings which containing edges $e_{1,2}$ and $f_{1,2}$,
and another is on the converse.
For a perfect matching $M\in \mathcal{M}(G_{n})$,
if $e_{1,2},f_{1,2}\in M$,
then $h_{1,2}$ is a unique $M$-alternating hexagon in the leftmost
phenanthrene consisting of three hexagons $s_{1,1},s_{1,2},h_{1,2}$, and the restriction $M^\prime$ of $M$ on the graph $G_{n-1}$ obtained by deleting
vertices of the leftmost phenanthrene and their incident edges from $G_n$
is a perfect matching of $G_{n-1}$.
By Lemma \ref{feh}, $f(G_{n},M)=f(G_{n-1},M^\prime)+1$.
On the other hand, if $e_{1,2},f_{1,2}\not\in M$,
then the restriction $M_1$ of $M$ on the leftmost diphenyl $N$
is a perfect matching of $N$,
and the restriction $M_2$ of $M$ on the successive subsystem $H_{n-1}$
is a perfect matching of $H_{n-1}$.
Therefore $f(G_{n},M)=f(N,M_1)+f(H_{n-1},M_2)$, see Fig. \ref{py1}.
By a similar deducing as Eq. $(\ref{eq11})$,
we can obtain the following formula

\begin{equation}\label{eq22}
F(G_n,x)=xF(G_{n-1},x)+4x^2F(H_{n-1}).
\end{equation}
Eq. $(\ref{eq11})$ minus Eq. $(\ref{eq22})$,
we have
\begin{equation}
\nonumber F(G_n,x)=F(H_n,x)-xF(H_{n-1},x),
\end{equation}
which implies
\begin{equation}
\nonumber F(G_{n-1},x)=F(H_{n-1},x)-xF(H_{n-2},x).
\end{equation}
Substituting this expression into Eq. $(\ref{eq11})$,
we can obtain Eq. (\ref{eq99}),
the proof is completed.
\end{proof}

\begin{theorem}\label{main4}{\bf .}
Let $H_n$ be a pyrene system with $n$ pyrene fragments.
Then
\begin{equation}
\nonumber F(H_n,x)=x^{n}\sum_{j=0}^{n}\sum_{i=\lceil\frac{j+n}{2}\rceil}^{n}
(-1)^{n-i}2^{2i+j-n}{i\choose n-i}{2i-n\choose j}x^{j}.
\end{equation}
\end{theorem}

\begin{proof}
For convenience, let $F_n:=F(H_n,x)$, then the generating function
of sequence $\{F_n\}_{n=0}^{\infty}$ is obtained as follow
\begin{align}\label{eq44}
\nonumber G(z)&=\sum_{n=0}^{\infty}F_{n}z^{n}=1+(4x^2+2x)z+\sum_{n=2}^{\infty}F_{n}z^{n}\\
\nonumber &=1+(4x^2+2x)z+\sum_{n=2}^{\infty}((4x^2+2x)F_{n-1}-x^2F_{n-2})z^{n}\\
\nonumber &=1+(4x^2+2x)z+(4x^2+2x)z(G(z)-1)-x^{2}z^{2}G(z)\\
\nonumber &=1+(4x^2+2x)zG(z)-x^{2}z^{2}G(z).
\end{align}

Therefore
\begin{eqnarray*}\label{eq55}
G(z)&=&\FS{1}{1-((4x^2+2x)z-x^2z^2)}\\
&=&\sum_{i=0}^{\infty}((4x^{2}+2x)z-x^{2}z^{2})^{i}\\
&=&\sum_{i=0}^{\infty}x^{i}z^{i}\sum_{j=0}^{i}{i \choose j}(4x+2)^{i-j}(-xz)^{j}\\
&=&\sum_{n=0}^{\infty}\sum_{i=\lceil\frac{n}{2}\rceil}^{n}(-1)^{n-i}
{i \choose n-i}(4x+2)^{2i-n}x^{n}z^{n},
\end{eqnarray*}
which implies
\begin{eqnarray*}\label{eq66}
F_n&=&x^{n}\sum_{i=\lceil\frac{n}{2}\rceil}^{n}(-1)^{n-i}{i\choose n-i}(4x+2)^{2i-n}\\
&=&x^{n}\sum_{i=\lceil\frac{n}{2}\rceil}^{n}(-1)^{n-i}{i\choose n-i}
\sum_{j=0}^{2i-n}2^{2i+j-n}{2i-n\choose j}x^{j}\\
&=&x^{n}\sum_{j=0}^{n}\sum_{i=\lceil\frac{j+n}{2}\rceil}^{n}
(-1)^{n-i}2^{2i+j-n}{i\choose n-i}{2i-n\choose j}x^{j}.
\end{eqnarray*}
The proof is completed.
\end{proof}

As a consequence,
the following corollary is immediate.

\begin{corollary}{\bf .}
Let $H_n$ be a pyrene system with $n$ pyrene fragments. Then
\begin{enumerate}
  \item $f(H_n)=n$;
  \item$F(H_n)=2n$;
  \item Spec$_f(H_n)=[n,2n]$.
\end{enumerate}
\end{corollary}

In the following we compute the degree of freedom of $H_n$,
and discuss its asymptotic behavior.
He and He \cite{He1} gave the following formula:
\begin{equation}\label{eq77}
\Phi(H_n)=6\Phi(H_{n-1})-\Phi(H_{n-2}),
\end{equation}
further we can obtain an general formula as follow:
\begin{equation}\label{eq78}
\Phi(H_n)=\frac{17-12\sqrt{2}}{16-12\sqrt{2}}(3-2\sqrt{2})^n+
\frac{17+12\sqrt{2}}{16+12\sqrt{2}}(3+2\sqrt{2})^n.
\end{equation}

\begin{theorem}\label{main5}{\bf .}
\begin{eqnarray}\label{eq79}
\nonumber IDF(H_n)&=&\FS{\sqrt{2}}{32}(3-2\sqrt{2})^n+\frac{7-5\sqrt{2}}{8}n(3-2\sqrt{2})^n
-\FS{\sqrt{2}}{32}(3+2\sqrt{2})^n\\
&~&+\frac{7+5\sqrt{2}}{8}n(3+2\sqrt{2})^n.
\end{eqnarray}
\end{theorem}
\begin{proof}
According to Eq. $(\ref{eq99})$,
\begin{eqnarray*}
\FS{d}{dx}F(H_n,x)&=&(8x+2)F(H_{n-1},x)+(4x^2+2x)\FS{d}{dx}F(H_{n-1},x)\\
&~&-2xF(H_{n-2},x)-x^2\frac{d}{dx}F(H_{n-2},x).
\end{eqnarray*}

For convenience, let $\Phi_n:=\Phi(H_n)$ and $IDF_n:=IDF(H_n)$,
then we have
\begin{eqnarray*}\label{eq881}
\nonumber IDF_n &=&\FS{d}{dx}F(H_n,x)\Big|_{x=1}\\
&=&6IDF_{n-1}-IDF_{n-2}+10\emph{$\Phi$}_{n-1}-2\Phi_{n-2}.
\end{eqnarray*}

So
\begin{eqnarray*}\label{eq882}
IDF_{n+1} &=&6IDF_{n}-IDF_{n-1}+10\Phi_{n}-2\Phi_{n-1},\\
IDF_{n+2} &=&6IDF_{n+1}-IDF_{n}+10\Phi_{n+1}-2\Phi_{n},
\end{eqnarray*}
by Eq. $(\ref{eq77})$, $\Phi_{n+1}=6\Phi_{n}-\Phi_{n-1}$ and $\Phi_{n}=6\Phi_{n-1}-\Phi_{n-2}$, which implies
\begin{eqnarray}\label{eq88}
\nonumber IDF_{n+2} &=&6IDF_{n+1}-IDF_{n}+10(6\Phi_{n}-\Phi_{n-1})-2(6\Phi_{n-1}-\Phi_{n-2})\\
\nonumber &=&6IDF_{n+1}-IDF_{n}+60\Phi_{n}-22\Phi_{n-1}+2\Phi_{n-2}\\
\nonumber &=&6IDF_{n+1}-IDF_{n}+6(6IDF_{n}-IDF_{n-1}+10\Phi_{n}-2\Phi_{n-1})
 -(6IDF_{n-1}\\
\nonumber &~&-IDF_{n-2}+10\Phi_{n-1}-2\Phi_{n-2})-36IDF_{n}+12IDF_{n-1}-IDF_{n-2}\\
&=&12IDF_{n+1}-38IDF_{n}+12IDF_{n-1}-IDF_{n-2}.
\end{eqnarray}
Therefore the homogeneous characteristics equation of recurrence formula $(\ref{eq88})$ is $x^4-12x^3 +38x^2-12x+1=0$,
and its roots are $x_1=x_2=3-2\sqrt{2}$, $x_3=x_4=3+2\sqrt{2}$.
Suppose the general solution of Eq. $(\ref{eq88})$ is $IDF_{n}=\lambda_{1}(3-2\sqrt{2})^n+\lambda_{2}n(3-2\sqrt{2})^n
+\lambda_{3}(3+2\sqrt{2})^n+\lambda_{4}n(3+2\sqrt{2})^n$.
According to the initial values $IDF_{3}=1036$,
$IDF_{4}=8068$, $IDF_{5}=58854$ and $IDF_{6}=411978$,
we can obtain $\lambda_{1}=\frac{\sqrt{2}}{32}$, $\lambda_{2}=\frac{7-5\sqrt{2}}{8}$,
$\lambda_{3}=-\frac{\sqrt{2}}{32}$ and $\lambda_{4}=\frac{7+5\sqrt{2}}{8}$,
so Eq. $(\ref{eq79})$ holds for $n\geq3$. In fact,
we can check that Eq. $(\ref{eq79})$ also holds for $n=0,1,2$, so the proof is completed.
\end{proof}

By Eq. $(\ref{eq78})$ and Eq. $(\ref{eq79})$,
the following result is obtained.

\begin{corollary}{\bf .}
Let $H_n$ be a pyrene system with $n$ pyrene fragments. Then
\begin{equation}
\nonumber \lim\limits_{n\rightarrow \infty}\FS{IDF(H_n)}{n\Phi(H_n)}=1+\FS{\sqrt{2}}{2}.
\end{equation}
\end{corollary}

\section{Anti-forcing polynomial of pyrene system}

The anti-forcing polynomial of a graph $G$
is defined as follow \cite{Hwang}:
\begin{equation}\label{eq101}
Af(G,x)=\sum_{M\in \mathcal{M}(G)} x^{af(G,M)}
=\sum_{i=af(G)}^{Af(G)} u_ix^{i},
\end{equation}
where $u_i$ is the number of perfect matchings of $G$
with the anti-forcing number $i$.

As a consequence, $\Phi(G)=Af(G,1)$, and the sum over the anti-forcing
numbers of all perfect matchings of $G$ equals $\frac{d}{dx}Af(G,x)\big|_{x=1}$.
If $G$ is a null graph or a graph with unique perfect matching,
then $Af(G,x)=1$.
Lemma \ref{minimax3} provides an approach for calculating the
anti-forcing number of a perfect matching of a pyrene system,
further we can obtain the following recursive formula.

\begin{theorem}\label{main7}{\bf .}
Let $H_n$ be the pyrene system with $n$ pyrene fragments.
Then
\begin{equation}\label{eq100}
Af(H_n,x)=(2x^3+2x^2+2x)Af(H_{n-1},x)-x^2Af(H_{n-2},x),
\end{equation}
where $n\geq2$, $Af(H_0,x)=1$ and $Af(H_1,x)=2x^3+2x^2+2x$.
\end{theorem}

\begin{proof}
First we divide $\mathcal{M}(H_n)$ in two subsets:
$\mathcal{M}_{f_{1,2}}^{e_{1,2}}(H_n)=\{M\in\mathcal{M}(H_n)\mid e_{1,2},f_{1,2}\in M\}$, $\mathcal{M}_{\bar{f}_{1,2}}^{\bar{e}_{1,2}}(H_n)=\{M\in\mathcal{M}(H_n)\mid e_{1,2},f_{1,2}\not\in M\}$. There are two cases to be considered.

\textbf{Case 1.} Suppose $e_{1,2}$ and $f_{1,2}$ both belong to $M$.
Then the restriction $M_1$ of $M$ on the leftmost pyrene fragment is a perfect
matching of it, and $h_{1,2}$ is an $M$-alternating hexagon ,see Fig. \ref{py}.

\textbf{Subcase 1.1.} If $p_{2,1}$ and $q_{2,1}$ both belong to $M$,
then the hexagons $s_{2,1}$ and $s_{2,2}$ both are $M$-alternating,
and the four hexagons $s_{1,1},s_{1,2},h_{1,2},h_{2,1}$
form a triphenylene whose perimeter $T$ is an $M$-alternating cycle,
and $\{h_{1,2},s_{2,1},s_{2,2},T\}$ is a non-crossing
compatible $M$-alternating set.
Note that the restriction $M^\prime$
of $M$ on the subsystem $H_{n-2}$ obtained by the removal of the first two pyrene fragments
is a perfect matching of $H_{n-2}$.
Let $\mathcal{A}^\prime$ be a maximum non-crossing
compatible $M^\prime$-alternating set of $H_{n-2}$,
by Lemma \ref{minimax3},
then $\{h_{1,2},s_{2,1},s_{2,2},T\}\cup \mathcal{A}^\prime$ is a
maximum non-crossing
compatible $M$-alternating set of $H_{n}$.
By Theorem \ref{minimax2},
$af(H_n,M)=4+af(H_{n-2},M^\prime)$.
Let $Y_1=\{M\in\mathcal{M}_{f_{1,2}}^{e_{1,2}}(H_n)|~p_{2,1},q_{2,1}\in M\}$,
by Eq. (\ref{eq101}),
\begin{equation}\label{c1}
\sum_{M\in Y_1}x^{af(H_n,M)}=\sum_{M^\prime\in \mathcal{M}(H_{n-2})}x^{4+af(H_{n-2},M^\prime)}
=x^4Af(H_{n-2},x).
\end{equation}

\textbf{Subcase 1.2.} If one of $p_{2,1},q_{2,1}$ does not
belong to $M$, then the perimeter of the triphenylene consisting of
the four hexagons $s_{1,1},s_{1,2},h_{1,2},h_{2,1}$ is not $M$-alternating.
Recall that $M_1\subseteq M$ is a perfect matching of the first pyrene fragment,
thus $M_2=M\setminus M_1$ is a perfect matching of the subgraph $G_{n-1}$ (see Fig. \ref{py1}).
By Lemma \ref{minimax3}, $af(H_n,M)=1+af(G_{n-1},M_2)$.
Let $X$ be a perfect matching of $G_{n-1}$.
Suppose $X$ contains edges $p_{2,1},q_{2,1}$,
then $s_{2,1}$ and $s_{2,2}$ both are $X$-alternating hexagons,
and $X_1=X\cap E(H_{n-2})$ is a perfect matching of the subsystem $H_{n-2}$
obtained by deleting the vertices of the leftmost diphenyl of $G_{n-1}$ and
their incident edges.
Note that Lemma \ref{minimax3} also holds for the auxiliary graph $G_n$,
and $h_{2,2}$ is not $X$-alternating, so $af(G_{n-1},X)=2+af(H_{n-2},X_1)$.
Let $\mathcal{M}^{p_{2,1}}_{q_{2,1}}(G_{n-1})=\{X\in  \mathcal{M}(G_{n-1})|p_{2,1},q_{2,1}\in X\}$, $Y_2=\mathcal{M}_{f_{1,2}}^{e_{1,2}}(H_n)\setminus Y_1$,
then
\begin{eqnarray}\label{c2}
\nonumber \sum_{M\in Y_2}x^{af(H_n,M)}&=&\sum_{M_2\in \mathcal{M}(G_{n-1})\setminus\mathcal{M}^{p_{2,1}}_{q_{2,1}}(G_{n-1})}x^{1+af(G_{n-1},M_2)}\\
\nonumber &=&x\Big{(}\sum_{X\in \mathcal{M}(G_{n-1})}x^{af(G_{n-1},X)}-\sum_{X\in \mathcal{M}^{p_{2,1}}_{q_{2,1}}(G_{n-1})}x^{af(G_{n-1},X)}\Big{)}\\
\nonumber &=& x\Big{(}Af(G_{n-1},x)-\sum_{X_1\in \mathcal{M}(H_{n-2})}x^{2+af(H_{n-2},X_1)}\Big{)}\\
&=& xAf(G_{n-1},x)-x^3Af(H_{n-2},x).
\end{eqnarray}

\textbf{Case 2.} Suppose $e_{1,2}$ and $f_{1,2}$ both are not in $M$,
then we can divide $\mathcal{M}_{\bar{f}_{1,2}}^{\bar{e}_{1,2}}(H_n)$
in two subsets $Y_3=\{M\in\mathcal{M}_{\bar{f}_{1,2}}^{\bar{e}_{1,2}}(H_n)|e_{2,1},f_{2,1}\in M\}$
and $Y_4=\{M\in\mathcal{M}_{\bar{f}_{1,2}}^{\bar{e}_{1,2}}(H_n)|e_{2,1},f_{2,1}\not\in M\}$.

\textbf{Subcase 2.1.} Suppose $M\in Y_3$, then $h_{2,1}$ must be an
$M$-alternating hexagon,
and the restrictions $M_1$ and $M_2$ of $M$ on the leftmost phenanthrene $L$
and the rightmost subsystem $H_{n-2}$ are perfect matchings of $L$ and $H_{n-2}$
respectively (see Fig. \ref{py}).
Let $\mathcal{A}^\prime$ be a maximum non-crossing compatible
$M_2$-alternating set of $H_{n-2}$.
Note that $M_1$ contains only  five distinct members,
we can divide $Y_3$ in five subsets:
$Y_{3,1}=\{M\in Y_3|p_{1,2},q_{1,2}\in M\}$,
$Y_{3,2}=\{M\in Y_3|p_{1,1},q_{1,1}\in M\}$,
$Y_{3,3}=\{M\in Y_3|e_{1,1},f_{1,1}\in M\}$,
$Y_{3,4}=\{M\in Y_3|p_{1,2}\in M,q_{1,2}\not\in M\}$,
$Y_{3,5}=\{M\in Y_3|p_{1,2}\not\in M,q_{1,2}\in M\}$.
If $M\in Y_{3,1}$,
then the four hexagons $h_{1,2},h_{2,1},s_{2,1},s_{2,2}$
form a triphenylene whose perimeter $T$ is an $M$-alternating cycle,
and $\{s_{1,1},s_{1,2},h_{2,1},T\}$ is a non-crossing
compatible $M$-alternating set.
By Lemma \ref{minimax3},
$\{s_{1,1},s_{1,2},h_{2,1},T\}\cup \mathcal{A}^\prime$
is a maximum non-crossing compatible
$M$-alternating set of $H_{n}$.
By Theorem \ref{minimax2},
$af(H_n,M)=4+af(H_{n-2},M_2)$,
which implies that $\sum_{M\in Y_{31}}x^{af(H_n,M)}=x^4Af(H_{n-2},x)$.
If $M\in Y_{3,2}$,
then $\{s_{1,1},s_{1,2},h_{1,1},h_{2,1}\}$ is an non-crossing compatible
$M$-alternating set, and $\{s_{1,1},s_{1,2},h_{1,1},h_{2,1}\}\cup \mathcal{A}^\prime$ is a maximum non-crossing compatible
$M$-alternating set of $H_{n}$.
By Theorem \ref{minimax2},
$af(H_n,M)=4+af(H_{n-2},M_2)$,
so $\sum_{M\in Y_{3,2}}x^{af(H_n,M)}=x^4Af(H_{n-2},x)$.
If $M\in Y_{3,3}$,
then $\{h_{1,1},h_{2,1}\}\cup \mathcal{A}^\prime$ is a maximum non-crossing compatible
$M$-alternating set of $H_{n}$.
By Theorem \ref{minimax2},
$af(H_n,M)=2+af(H_{n-2},M_2)$,
we have $\sum_{M\in Y_{3,3}}x^{af(H_n,M)}=x^2Af(H_{n-2},x)$.
If $M\in Y_{3,4}$ or $M\in Y_{3,5}$,
then $\{s_{1,1},s_{1,2},h_{2,1}\}\cup \mathcal{A}^\prime$ is
a maximum non-crossing compatible
$M$-alternating set of $H_{n}$.
By Theorem \ref{minimax2}, $af(H_n,M)=3+af(H_{n-2},M_2)$,
thus $\sum_{M\in Y_{3,4}}x^{af(H_n,M)}+\sum_{M\in Y_{3,5}}x^{af(H_n,M)}=2x^3Af(H_{n-2},x)$.
Finally, we have
\begin{equation}\label{c3}
\sum_{M\in Y_{3}}x^{af(H_n,M)}=\sum_{j=1}^{5}\sum_{M\in Y_{3,j}}x^{af(H_n,M)}
=(2x^4+2x^3+x^2)Af(H_{n-2},x).
\end{equation}

\textbf{Subcase 2.2.} If $M\in Y_4$,
then the common vertical edge $d$ of $h_{1,2}$ and $h_{2,1}$
belongs to $M$,
and the restrictions $M_1$ and $M_2$ of $M$
on the leftmost pyrene fragment $H_1$ and the rightmost subsystem
$H_{n-1}$ are perfect matchings of $H_1$ and $H_{n-1}$ respectively (see Fig. \ref{py}).
We divide $\mathcal{M}(H_1)$ in two subsets:
$\mathcal{M}_d(H_1)=\{M_1\in\mathcal{M}(H_1)|d\in M_1\}$,
$\mathcal{M}_{\bar{d}}(H_1)=\{M_1\in\mathcal{M}(H_1)|d\not\in M_1\}$.
Note that $\mathcal{M}_{\bar{d}}(H_1)$ contains only one perfect matching $M_1^\prime$ of $H_1$, and $h_{1,2}$ is the unique $M_1^\prime$-alternating hexagon in $H_1$,
so $af(H_1,M_1^\prime)=1$,
we have
\begin{eqnarray}\label{c4}
\nonumber\sum_{M_1\in \mathcal{M}_d(H_1)}x^{af(H_1,M_1)}&=&\sum_{M_1\in \mathcal{M}(H_1)}x^{af(H_1,M_1)}~-\sum_{M_1^\prime\in \mathcal{M}_{\bar{d}}(H_1)}x^{af(H_1,M_1^\prime)}\\
\nonumber &=&Af(H_1,x)-x\\
&=&2x^3+2x^2+x.
\end{eqnarray}
We also divide $\mathcal{M}(H_{n-1})$ in two subsets:
$\mathcal{M}_d(H_{n-1})=\{M_2\in\mathcal{M}(H_{n-1})|d\in M_2\}$,
$\mathcal{M}_{\bar{d}}(H_{n-1})=\{M_2\in\mathcal{M}(H_{n-1})|d\not\in M_2\}$.
Suppose $M_2\in \mathcal{M}_{\bar{d}}(H_{n-1})$,
then $e_{2,1},f_{2,1}\in M_2$ and $h_{2,1}$ is an $M_2$-alternating hexagon,
and the restriction $M_2^\prime$ of $M_2$ on the rightmost
subsystem $H_{n-2}$ is a perfect matching of $H_{n-2}$.
Let $\mathcal{A}^\prime$ be a maximum non-crossing compatible
$M_2^\prime$-alternating set of $H_{n-2}$.
Then $\mathcal{A}^\prime\cup \{h_{2,1}\}$ is
a maximum non-crossing compatible
$M_2$-alternating set of $H_{n-1}$.
Thus $af(H_{n-1,M_2})=1+af(H_{n-2},M_2^\prime)$,
we have

\begin{eqnarray}\label{c5}
\nonumber \sum_{M_2\in \mathcal{M}_{d}(H_{n-1})}x^{af(H_{n-1},M_2)}&=&
\sum_{M_2\in \mathcal{M}(H_{n-1})}x^{af(H_{n-1},M_2)}~-
\sum_{M_2\in \mathcal{M}_{\bar{d}}(H_{n-1})}x^{af(H_{n-1},M_2)}\\
\nonumber &=& Af(H_{n-1},x)-\sum_{M_2^\prime\in \mathcal{M}(H_{n-2})}x^{1+af(H_{n-2},M_2^\prime)}\\
 &=& Af(H_{n-1},x)-xAf(H_{n-2},x).
\end{eqnarray}

Recall that $d$ is the common edge of $h_{1,2}$ and $h_{2,1}$,
for any $M\in Y_4$,
then $M=M_1\cup M_2$,
where $M_1$ is a perfect matching of the first pyrene fragment $H_1$
and $M_2$ is a perfect matching of the rightmost subsystem $H_{n-1}$,
and $\{d\}=M_1\cap M_2$.
By Theorem \ref{minimax2} and Lemma \ref{minimax3},
we have $af(H_n,M)=af(H_1,M_1)+af(H_{n-1},M_2)$.
According to Eqs. (\ref{c4}) and (\ref{c5}),
we have
\begin{eqnarray}\label{c6}
\nonumber\sum_{M\in Y_4}x^{af(H_n,M)}&=&\sum_{M_1\in \mathcal{M}_d(H_1),M_2\in \mathcal{M}_{d}(H_{n-1})}x^{af(H_1,M_1)+af(H_{n-1},M_2)}\\
\nonumber &=& \Big{(}\sum_{M_1\in \mathcal{M}_d(H_1)}x^{af(H_1,M_1)}\Big{)}\Big{(}\sum_{M_2\in \mathcal{M}_{d}(H_{n-1})}x^{af(H_{n-1},M_2)}\Big{)}\\
\nonumber &=&(2x^3+2x^2+x)(Af(H_{n-1},x)-xAf(H_{n-2},x))\\
\nonumber&=&(2x^3+2x^2+x)Af(H_{n-1},x)-(2x^4+2x^3+x^2)Af(H_{n-2},x).\\
~~
\end{eqnarray}

By Eqs. (\ref{c1}), (\ref{c2}), (\ref{c3}) and (\ref{c6}),
we obtain a recursive relation as below:
\begin{eqnarray}\label{c7}
\nonumber Af(H_n,x)&=&\sum_{M\in \mathcal{M}(H_n)}x^{af(H_n,M)}\\
\nonumber&=&\sum_{M\in Y_1}x^{af(H_n,M)}+\sum_{M\in Y_2}x^{af(H_n,M)}+\sum_{M\in Y_3}x^{af(H_n,M)}+\sum_{M\in Y_4}x^{af(H_n,M)}\\
\nonumber &=& (2x^3+2x^2+x)Af(H_{n-1},x)+(x^4-x^3)Af(H_{n-2},x)+xAf(G_{n-1},x).\\
~~
\end{eqnarray}

By a similar discussion as above,
we can prove the following recursive formula for the auxiliary
graph $G_n$ (see Fig. \ref{py1}),
\begin{equation}\label{c8}
Af(G_n,x)=(x^3+3x^2)Af(H_{n-1},x)+(x^4-x^3)Af(H_{n-2},x)+xAf(G_{n-1},x).
\end{equation}
Eq. (\ref{c7}) subtracts Eq. (\ref{c8}),
we have
\begin{equation}\label{c9}
\nonumber Af(G_n,x)=Af(H_n,x)-(x^3-x^2+x)Af(H_{n-1},x),
\end{equation}
so
\begin{equation}\label{c10}
\nonumber Af(G_{n-1},x)=Af(H_{n-1},x)-(x^3-x^2+x)Af(H_{n-2},x).
\end{equation}
Substituting this expression into Eq. (\ref{c7}),
we can obtain the Eq. (\ref{eq100}),
the proof is completed.
\end{proof}

By theorem \ref{main7}, we can
obtain an explicit expression as below.
\begin{theorem}\label{main6}{\bf .}
Let $H_n$ be the pyrene system with $n$ pyrene fragments.
Then
\begin{equation}\label{eq103}
Af(H_n,x)=x^{n}\sum_{l=0}^{2n}\sum_{i=\lceil\frac{l+2n}{4}\rceil}^{n}
\sum_{j=\lceil\frac{l}{2}\rceil}^{l}
(-1)^{n-i}2^{2i-n}{i\choose 2i-n}
{2i-n\choose j}{j\choose l-j} x^{l}.
\end{equation}
\end{theorem}

\begin{proof}
Let $A_n:= Af(H_n,x)$, then the generating function of sequence
$\{A_n\}_{n=0}^{\infty}$ is
\begin{eqnarray}\label{eq104}
\nonumber G(t)&=&\sum_{n=0}^{\infty}A_nt^n=1+(2x^3+2x^2+2x)t+\sum_{n=2}^{\infty}A_nt^n\\
\nonumber &=& 1+(2x^3+2x^2+2x)t+\sum_{n=2}^{\infty}((2x^3+2x^2+2x)A_{n-1}-x^2A_{n-2})t^n\\
\nonumber
&=&1+(2x^3+2x^2+2x)t\sum_{n=0}^{\infty}A_{n}t^n-x^2t^2\sum_{n=0}^{\infty}A_{n}t^n\\
\nonumber &=& 1+(2x^3+2x^2+2x)tG(t)-x^2t^2G(t).
\end{eqnarray}
So
\begin{eqnarray}
\nonumber G(t)&=&\FS{1}{1-((2x^3+2x^2+2x)t-x^2t^2)}\\
\nonumber &=&\sum_{i=0}^{\infty}((2x^3+2x^2+2x)t-x^2t^2)^i\\
\nonumber &=&\sum_{i=0}^{\infty}\sum_{j=0}^{i}{i\choose j}(2x^3+2x^2+2x)^jt^j(-x^2t^2)^{i-j}\\
\nonumber &=&\sum_{i=0}^{\infty}\sum_{n=i}^{2i}(-1)^{n-i}2^{2i-n}{i\choose 2i-n}
(x^2+x+1)^{2i-n}x^{n}t^{n}\\
\nonumber &=&\sum_{n=0}^{\infty}\sum_{i=\lceil\frac{n}{2}\rceil}^{n}(-1)^{n-i}2^{2i-n}{i\choose 2i-n}
(x^2+x+1)^{2i-n}x^{n}t^{n},
\end{eqnarray}
we have
\begin{eqnarray}\label{eq105}
\nonumber Af(H_n,x)&=&x^{n}\sum_{i=\lceil\frac{n}{2}\rceil}^{n}(-1)^{n-i}2^{2i-n}{i\choose 2i-n}
(x^2+x+1)^{2i-n}\\
\nonumber &=& x^{n}\sum_{i=\lceil\frac{n}{2}\rceil}^{n}(-1)^{n-i}2^{2i-n}{i\choose 2i-n}
\sum_{j=0}^{2i-n}{2i-n\choose j}x^{j}\sum_{k=0}^{j}{j\choose k}x^{k}\\
\nonumber &=& x^{n}\sum_{i=\lceil\frac{n}{2}\rceil}^{n}(-1)^{n-i}2^{2i-n}{i\choose 2i-n}
\sum_{j=0}^{2i-n}\sum_{l=j}^{2j}{2i-n\choose j}{j\choose l-j} x^{l}\\
\nonumber &=& x^{n}\sum_{l=0}^{2n}\sum_{i=\lceil\frac{l+2n}{4}\rceil}^{n}
\sum_{j=\lceil\frac{l}{2}\rceil}^{l}
(-1)^{n-i}2^{2i-n}{i\choose 2i-n}
{2i-n\choose j}{j\choose l-j} x^{l}.
\end{eqnarray}
\end{proof}

According to Theorem \ref{main6}, the following corollary is immediate.

\begin{corollary}{\bf .}
Let $H_n$ be a pyrene system with $n$ pyrene fragments.
Then
\begin{enumerate}
  \item $af(H_n)=n$;
  \item $Af(H_n)=3n$;
  \item Spec$_{af(H_n)}=[n,3n]$.
\end{enumerate}
\end{corollary}

In the following, we will calculate the sum over the anti-forcing
numbers of all perfect matchings of $H_n$,
and investigate its asymptotic behavior.

\begin{theorem}\label{main8}{\bf .}
The sum over the anti-forcing
numbers of all perfect matchings of $H_n$ is
\begin{eqnarray}\label{eq106}
\nonumber \FS{d}{dx}Af(H_n,x)\big|_{x=1}&=&\FS{3\sqrt{2}}{64}(3-2\sqrt{2})^n+\frac{17-12\sqrt{2}}{16}n(3-2\sqrt{2})^n
-\FS{3\sqrt{2}}{64}(3+2\sqrt{2})^n\\
&~&+\frac{17+12\sqrt{2}}{16}n(3+2\sqrt{2})^n.
\end{eqnarray}
\end{theorem}


\begin{proof}
By Theorem \ref{main7},
\begin{eqnarray}\label{eq107}
\nonumber \FS{d}{dx}Af(H_n,x)&=&(6x^2+4x+2)Af(H_{n-1},x)+(2x^3+2x^2+2x)\FS{d}{dx}Af(H_{n-1},x)\\
&~&-2xAf(H_{n-2},x)-x^2\FS{d}{dx}Af(H_{n-2},x).
\end{eqnarray}
For convenience, let $\Phi_n:=\Phi(H_n)$ and $AF_n:=\frac{d}{dx}Af(H_n,x)\big|_{x=1}$,
by Eq. (\ref{eq107}), we have
\begin{eqnarray}\label{eq108}
AF_n=6AF_{n-1}-AF_{n-2}+12\Phi_{n-1}-2\Phi_{n-2}.
\end{eqnarray}
By Eq. $(\ref{eq77})$, $\Phi_n=6\Phi_{n-1}-\Phi_{n-2}$,
so $AF_n=6AF_{n-1}-AF_{n-2}+2\Phi_{n}$,
which implies $2\Phi_{n}=AF_n-6AF_{n-1}+AF_{n-2}$.
Therefore $2\Phi_{n-1}=AF_{n-1}-6AF_{n-2}+AF_{n-3}$
and $2\Phi_{n-2}=AF_{n-2}-6AF_{n-3}+AF_{n-4}$, substituting
them into Eq. (\ref{eq108}), we obtain the following recurrence formula
\begin{equation}\label{eq109}
AF_n=12AF_{n-1}-38AF_{n-2}+12AF_{n-3}-AF_{n-4}.
\end{equation}
Note that recurrence formulas $(\ref{eq88})$ and $(\ref{eq109})$
have the same homogeneous characteristics equation,
so the general solution of Eq. $(\ref{eq109})$ is $AF_{n}=\lambda_{1}(3-2\sqrt{2})^n+\lambda_{2}n(3-2\sqrt{2})^n
+\lambda_{3}(3+2\sqrt{2})^n+\lambda_{4}n(3+2\sqrt{2})^n$.
By the initial values $AF_{5}=70956$,
$AF_{6}=496794$, $AF_{7}=3380640$ and $AF_{8}=22531256$,
we have $\lambda_{1}=\frac{3\sqrt{2}}{64}$, $\lambda_{2}=\frac{17-12\sqrt{2}}{16}$,
$\lambda_{3}=-\frac{3\sqrt{2}}{64}$ and $\lambda_{4}=\frac{17+12\sqrt{2}}{16}$,
so Eq. $(\ref{eq106})$ holds for $n\geq5$.
We can check that Eq. (\ref{eq106}) also holds for $n=0,1,2,3,4$,
the proof is completed.
\end{proof}

By Eq. $(\ref{eq78})$ and Eq. $(\ref{eq106})$,
we can prove the following corollary.

\begin{corollary}{\bf .}
Let $H_n$ be a pyrene system with $n$ pyrene fragments. Then
\begin{equation}
\nonumber \lim\limits_{n\rightarrow \infty}\FS{AF_n}{n\Phi_n}=1+\FS{3\sqrt{2}}{4}.
\end{equation}
\end{corollary}

\cite{}

\end{document}